\documentclass{amsart}
\usepackage{amsmath}
\usepackage{amssymb}
\usepackage{amsfonts}
\usepackage{mathabx}
\usepackage{xcolor}
\usepackage[utf8]{inputenc}
\usepackage{fullpage}
\usepackage{mathptmx}
\usepackage{enumerate}
\usepackage{bm}
\usepackage[shortlabels]{enumitem}

\usepackage{hyperref}
\hypersetup{pdfstartview=FitH, colorlinks=True, linkcolor=blue, citecolor=blue, urlcolor=blue}

\newcommand{\NN}{\mathbb{N}}
\newcommand{\ZZ}{\mathbb{Z}}
\newcommand{\QQ}{\mathbb{Q}}
\newcommand{\RR}{\mathbb{R}}
\newcommand{\CC}{\mathbb{C}}

\newcommand{\eps}{\varepsilon}

\newcommand{\diam}{\text{diam}}
\newcommand{\proj}{\text{proj}}

\newtheorem{theorem}{Theorem}
\newtheorem{lemma}{Lemma}
\newtheorem{prop}{Proposition}
\newtheorem{corollary}{Corollary}

\theoremstyle{definition}
\newtheorem{remark}{Remark}

\title{Generic power series on subsets of the unit disk}

\author{Bal\'azs Maga}
\author{P\'eter Maga}

\address{E\"otv\"os Lor\'and University, Department of Analysis, P\'azm\'any P\'eter s\'et\'any 1/C, Budapest, H–1117 Hungary}
\email{magab@cs.elte.hu}
\address{MTA Alfr\'ed R\'enyi Institute of Mathematics, POB 127, Budapest H-1364, Hungary}\email{magapeter@gmail.com}

\thanks{The first author was supported by the \'UNKP-20-3 New National Excellence Program of the Ministry for Innovation and Technology from the source of the National Research, Development and Innovation Fund, and by the Hungarian National Research, Development and Innovation Office–NKFIH, Grant 124003. The second author was supported by the Premium Postdoctoral Fellowship of the Hungarian Academy of Sciences, the MTA Rényi Intézet Automorphic Research Group, and by the Hungarian National Research, Development and Innovation Office–NKFIH, Grants FK~135218, K~119528.}

\begin{document}

\begin{abstract}
In this paper, we examine the boundary behaviour of the generic power series $f$ with coefficients chosen from a fixed bounded set $\Lambda$ in the sense of Baire category. Notably, we prove that for any open subset $U$ of the unit disk $D$ with a non-real boundary point on the unit circle, $f(U)$ is a dense set of $\CC$. As it is demonstrated, this conclusion does not necessarily hold for arbitrary open sets accumulating to the unit circle. To complement these results, a characterization of coefficient sets having this property is given.
\end{abstract}
\maketitle

\section{Introduction}

Let $\Lambda\subseteq\CC$ be a bounded subset with at least two elements, endowed with its usual subspace topology. Moreover, assume that the product space
\[
\Omega: = \Omega_\Lambda: = \bigtimes_{n=0}^{\infty} \Lambda
\]
is a Baire space. Due to Alexandrov's theorem (e.g. \cite[p.~408]{Kur}), this holds, for example, if $\Lambda$ is $G_\delta$. These general conditions on $\Lambda$ will be assumed throughout the paper. 

For any $\bm{\lambda} = (\lambda_n)_{n=0}^{\infty} \in \Omega$ we can define the power series
\[
f(z): = f_{\bm{\lambda}}(z):=\sum_{n=0}^{\infty}\lambda_n z^n.
\]
The resulting function is clearly holomorphic in the open unit disk $D:=\{|z|< 1\}$. Roughly speaking, we are interested in the generic behaviour of $f$ in terms of Baire category near the boundary $\partial D = S$. The genericity is understood as follows: if a property holds for a set of power series corresponding to a residual set of configurations in $\Omega$, we say it is generic.

This work is a direct continuation of our previous paper \cite{MM}, in which we investigated the typical boundary behaviour of real power series with coefficients chosen from a finite set $\Lambda$. (Actually, the set of coefficients was denoted by $D$ in that paper, we opted to introduce this notational modification as the symbol $D$ is customarily preserved for the unit disk in this setup.) While the probabilistic aspects of the problem was our main focus (considering the uniform distribution over $\Lambda$), we proved straightforward results in terms of Baire category as well (\cite[Theorem~3]{MM}). Notably, if $\Lambda$ has both positive and negative elements, for the generic power series $f$ we have $\limsup_{1-}f=+\infty$ and $\liminf_{1-}f=-\infty$, while if $\Lambda\subseteq[0, +\infty)$ (resp. $\Lambda\subseteq(-\infty, 0]$), we have $\lim_{1-}f=+\infty$ (resp. $\lim_{1-}f=-\infty$). It was natural to consider the same problem in the complex setup as well, which is a more natural habitat of power series. (We note that while those results were stated for finite $\Lambda$ exclusively, the proofs can be generalized in a straightforward manner for any $\Lambda$ for which $\bigtimes_{n=0}^{\infty} \Lambda$ is a Baire space.)

A direct prelude to our results is given in \cite{BS}. Even though the main focus of that paper is the theory of random power series, where it presents spectacular results about natural boundaries without assuming independence, it contains the following
result relevant to our setting:

\begin{prop}[{\cite[Theorem~1.9]{BS}}] \label{prop:BS} For generic $\bm{\lambda}=(\lambda_n)_{n=0}^{\infty}\in\Omega$, the power series $f_{\bm{\lambda}}$ has a natural boundary on $\partial D$. That is it cannot be analytically continued through any of the boundary points of the convergence domain.
\end{prop}

Another predecessor of this line of research is \cite{Kah}, in which results are stated and proved about generic complex power series. For a more detailed historical summary of the topic we also refer to that paper. We recall that while the problem of the ``general behaviour'' dates back to Borel, the more thorougly examined probabilistic question was worked out by several authors, as presented in \cite{Kah2}. In terms of Baire category, the first results were provided in \cite{KS}. The setup of \cite{KS} slightly differs from ours, as the topological vector space $H(D)$ in which Baire category is investigated is the space of all functions which are holomorphic in $D$ with the topology of locally uniform convergence. (The $\Omega$ we consider corresponds to a subspace of it.)  Their main results stated that generically, $S$ is a natural boundary and $f(D)=\CC$. These results were generalized by \cite{Kah}, in particular:

\begin{prop}[{\cite[Proposition~3.2]{Kah}}] \label{prop:Kah} Assume that $U \subseteq D$ is open and $\overline U \cap S \neq \emptyset$. Then for generic $f\in H(D)$ we have $f(U)=\CC$ generically. (For a set $B\subseteq \CC$, its closure is denoted by $\overline{B}$ throughout the paper.)
\end{prop}

Our main goal is strengthening Proposition \ref{prop:BS} in a similar manner. The proof of Proposition \ref{prop:Kah} relies on Runge's theorem on polynomial approximation, which is out of reach in our setup that poses restrictions on the permissible holomorphic functions. Consequently, when one would like to verify similar results for $\Omega$, different techniques are required. As we will see, this leads to somewhat weaker theorems: roughly, we could only verify that $f(U)$ is an open, dense set instead of being equal to $\CC$. Our first main result is the following.

\begin{theorem} \label{thm:not_-1_1_arguments}
Assume that $U \subseteq D$ is open with an accumulation point $\zeta\in S$ with non-vanishing imaginary part. Then for generic $\bm{\lambda}=(\lambda_n)_{n=0}^{\infty}\in\Omega$, the image $f(U)\subseteq \CC$ is dense and open.
\end{theorem}

As we will see below, the conclusion of Theorem \ref{thm:not_-1_1_arguments} does not hold for all open sets accumulating to $S$. In order to allow 1 or $-1$ to be the accumulation point on the boundary, certain conditions on $\Lambda$ should be introduced. The necessary and sufficient conditions are summarised by our other theorems. By a real line, we mean a one-dimensional real affine subspace of $\mathbb{C}$, while by a real half-plane we mean one of the two components of the complement of a real line. We refer to its closure as a closed real half-plane.

\begin{theorem} \label{thm:not_1_arguments}
If $\Lambda$ is not contained by a real line, and $U \subseteq D$ is open with $-1 \in \overline{U}$, then for generic $\bm{\lambda}=(\lambda_n)_{n=0}^{\infty}\in\Omega$, the image $f_{\bm{\lambda}}(U)\subseteq \CC$ is dense and open. If $\Lambda$ is contained by a real line, there exists an open $U \subseteq D$ with $-1 \in \overline{U}$ for which $f_{\bm{\lambda}}(U)$ evades a closed real half-plane, regardless of the choice of $\bm{\lambda}$.
\end{theorem}

\begin{theorem} \label{thm:all_arguments}
If $\Lambda$ is not contained by a closed real half-plane of the form $\{z: \text{ } \alpha \leq \arg z \leq \alpha + \pi\}$, and $U \subseteq D$ is open with $1 \in \overline{U}$, then for generic $\bm{\lambda}=(\lambda_n)_{n=0}^{\infty}\in\Omega$, the image $f_{\bm{\lambda}}(U)\subseteq \CC$ is dense and open. If $\Lambda$ is contained by a closed real half-plane of the form $\{z: \text{ } \alpha \leq \arg z \leq \alpha + \pi\}$, then there exists an open $U \subseteq D$ with $1 \in \overline{U}$ for which $f_{\bm{\lambda}}(U)$ evades a closed real half-plane, regardless of the choice of $\bm{\lambda}$.
\end{theorem}

We note that the openness of $f(U)$ for generic $f$ is a trivial consequence of the open mapping theorem for analytic functions in each of the cases, as such $f$ are non-constant. This implies that the real task in these questions is proving the density of the images.

We fix some notation for the rest of the paper. Throughout, $D_r=\{|z| < r\}$ for general disks centered at the origin, while its boundary is denoted by $S_r$ (that is, $D=D_1$, $S=S_1$.) We use the notation $\proj_n(G)$ for the projection of $G\subseteq (\lambda_n)_{n=0}^{\infty} \in \Omega$ to the $n$th coordinate of the product.

For any $\zeta\in S$ and any $N\in\NN$, define
\[
\{0,1\}_N[\zeta]:=\left\{\sum_{j=N}^{\infty} a_j \zeta^j: a_j\in\{0,1\},\ a_j=0\text{ with finitely many exceptions}\right\}.
\]
A terminology we are going to use frequently is the following. For sets $A,B\subseteq\CC$ and some $\eps>0$, we say that $A$ is an $\eps$-net of $B$, if for any $w\in B$, there exists some $z\in A$ such that $|z-w|<\eps$.

Below we will also use the notation $e(x)=e^{2\pi i x}$ for $x\in\CC$.

\textbf{Acknowledgement.} We are grateful to the anonymous referee for their careful reading of the manuscript and their suggestions to improve the exposition.

\section{Preliminary statements}

We fix the following notation: for any $c\in\CC$, set
\[
\Lambda+c:=\{\lambda+c:\lambda\in\Lambda\},\qquad c\Lambda:=\{c\lambda:\lambda\in\Lambda\},
\]
and 
\[\bm{\lambda}+c=(\lambda_n)_{n=0}^{\infty}+c := (\lambda_n + c)_{n=0}^{\infty},\qquad c\bm{\lambda}=c(\lambda_n)_{n=0}^{\infty}c: = (c\lambda_n)_{n=0}^{\infty}.\]
Set also
\[
f_{\Lambda}(U):=\bigcup_{\bm{\lambda}\in\Omega} f_{\bm{\lambda}}(U).
\]

Our first lemma concerns the effect on $f_{\bm{\lambda}}(U)$ of certain modifications on $\Lambda$, enabling us to circumvent some of the technical burden in the proof of our theorems through replacing arbitrary $\Lambda$'s by simpler ones.

\begin{lemma} \label{lemma:translate_Lambda} 
\begin{enumerate}[(a)]
\item\label{lemma:translate_mult} Let $U\subseteq D$ be open. Assume that $f_{\bm{\lambda}}(U)$ is dense in $\CC$. Then for any $0\neq c\in\CC$, $f_{c\bm{\lambda}}(U)$ is also dense in $\CC$.
\item\label{lemma:translate_add_shrink} Let $(U_k)_{k=1}^{\infty}$ be a shrinking sequence of open subsets of $D$ such that $\diam(U_k)\to 0$ and none of them accumulates to $1$. Moreover, assume that all of the sets $f_{\bm{\lambda}}(U_k)$ are dense in $\CC$. Then for any $c\in\CC$, all of the sets $f_{\bm{\lambda} + c}(U_k)$ are dense in $\mathbb{C}$. 
\item\label{lemma:translate_add_evade} Assume that $f_{\Lambda}(U)$ evades a closed real half-plane for some open $U\subseteq D$ with $1\notin \overline{U}$.
Then for any $c\in\CC$, the same holds for $f_{\Lambda+c}(U)$.
\item\label{lemma:genericity_effect} If the assumption of \ref{lemma:translate_mult} (resp. \ref{lemma:translate_add_shrink}) is a generic property in $\Omega_\Lambda$, then its implication is a generic property in $\Omega_{c\Lambda}$ (resp. $\Omega_{\Lambda + c}$).
\end{enumerate}
\end{lemma}

\begin{remark} \label{remark:ref_comment}
The assumption of Lemma~\ref{lemma:translate_Lambda}~\ref{lemma:translate_add_shrink} looks a bit complicated and one may wonder if it can be formulated in a much simpler way, akin to Lemma~\ref{lemma:translate_Lambda}~\ref{lemma:translate_mult}. Notably, one can intuitively believe that assuming $\overline{U}\cap S\setminus \{1\} \neq \emptyset$, the density of $f_{\bm{\lambda}}(U)$ in $\CC$ implies that $f_{\bm{\lambda}+c}(U)$ is also dense in $\CC$ for any $c\in\CC$. (This formulation would be directly usable in the proof of Theorem \ref{thm:not_-1_1_arguments} and \ref{thm:not_1_arguments} to translate $\Lambda$.) However, this claim is false: a simple counterexample is given by 
$$\Lambda=\{0, 1\},\qquad U=D$$
and
$$f_{\bm{\lambda}}(z)=\sum_{k=0}^{\infty}z^{2k+1}=\frac{z}{1-z^2}.$$
Indeed, $f_{\bm{\lambda}}(z)=\alpha$ leads to a quadratic equation such that its roots has product -1. Consequently, it has a root in $\overline{D}$, which quickly yields the density of $f_{\bm{\lambda}}(D)$. However, 
$$f_{\bm{\lambda}-\frac{1}{2}}(z)=f_{\bm{\lambda}}(z) - \frac{1}{2}\cdot\frac{1}{1-z} = -\frac{1}{2(1+z)},$$
for which $f_{\bm{\lambda}-\frac{1}{2}}(U)$ is clearly not dense in $\mathbb{C}$.
\end{remark}

\begin{proof}[Proof of Lemma \ref{lemma:translate_Lambda}]
Statement \ref{lemma:translate_mult} follows trivially from the relation $f_{c\lambda}(U)=cf_{\lambda}(U)$. 

To prove \ref{lemma:translate_add_shrink} and \ref{lemma:translate_add_evade}, consider the mapping
\begin{equation} \label{eq:translation_correspondence}
f_{\bm{\lambda}} \mapsto f_{\bm{\lambda} + c},\qquad  f_{\bm{\lambda} + c}(z)= f_{\bm{\lambda}}(z) + \frac{c}{1-z} = f_{\bm{\lambda}}(z) + g(z).
\end{equation}
Now if $(U_k)_{k=1}^{\infty}$ is a sequence satisfying the conditions of \ref{lemma:translate_add_shrink}, we clearly have $\diam(g(U_k))\to 0$, which easily implies the statement due to \eqref{eq:translation_correspondence}.


As for \ref{lemma:translate_add_evade}, observe that in \eqref{eq:translation_correspondence}, $g(U)$ is bounded under the assumptions. Therefore if $f_{\Lambda}(U)$ is a subset of a closed real half-plane, then so is $f_{\Lambda+c}(U)$. 

Finally, let us observe that \ref{lemma:genericity_effect} is obvious.
\end{proof}

The following three lemmata serve as a preparation to the proof of Theorem~\ref{thm:not_-1_1_arguments}.

\begin{lemma} \label{lemma:dense_set_sum}
Assume that $(H_j)_{j=1}^{\infty}$ is a sequence of dense subsets of $S$. Then $\bigcup_{k=1}^{\infty}\sum\limits_{j=1}^{k}H_j$ is dense in $\mathbb{C}$, where $\sum\limits_{j=1}^{k}H_j$ denotes the Minkowski sum of $H_1, ..., H_k$.
\end{lemma}

\begin{proof}
The proof follows from three simple observations:
\begin{itemize}
\item for any $k\geq{1}$, $$\overline{\left(\sum_{j=1}^{k}H_j\right)} = \sum_{j=1}^{k}\overline{H_j};$$
\item $S_1 + S_1 = \overline{D_2}$;
\item $D_r + S_{r'} = D_{r+r'}$ for $0<r'<r$.
\end{itemize}
Putting together these claims yields that $$\overline{\left(\sum_{j=1}^{k}H_j\right)}=\overline{D_k}$$ for $k\geq 2$. Consequently, $$\overline{\left(\bigcup_{k=1}^{\infty}\sum\limits_{j=1}^{k}H_j\right)}=\mathbb{C}$$ clearly holds.
\end{proof}

\begin{lemma}\label{lemma:circle_group_ae_Z_density}
Let $\zeta\in S$ be different from $\pm 1, \pm i, \pm \omega, \pm\omega^2$, where $i=e(1/4)$, $\omega=e(1/3)$. Then for any $N\in\NN$, $\{0,1\}_N[\zeta]$
is dense in $\CC$, i.e. $\{0,1\}_N[\zeta]$ is an $\eps$-net of $\CC$ for any $\eps>0$.
\end{lemma}


\begin{proof}
If $\zeta=e(x)$ with $x\in\RR\setminus\QQ$, then the statement follows simply from Lemma~\ref{lemma:dense_set_sum}. Indeed, if $w\in\CC$ and $\eps>0$ be given, our goal is to approximate $w$ with error smaller than $\eps$ with a finite sum of the form given in the statement. Setting \[H_1:=H_2:=\ldots:=\{\zeta^N,\zeta^{N+1},\zeta^{N+2},\ldots\}\] in Lemma~\ref{lemma:dense_set_sum}, we obtain a certain $z:=\zeta^{n_1}+\ldots+\zeta^{n_k}$ such that $|z-w|<\eps$. Possibly there are repetitions among the $n_j$'s, but any $\zeta^{n_j}$ can be replaced with another $\zeta^{n_{j'}}$ on the cost of an arbitrarily small error, which altogether verifies the claim.

If $\zeta=e(x)$ with $x\in\QQ$, a root of unity different from $\pm 1, \pm i, \pm \omega, \pm \omega^2$, then its degree over $\QQ$ is greater than $2$. In particular, $\zeta+\zeta^{-1}\in\RR\setminus \QQ$, and since $\zeta,\zeta^{-1}$ can be obtained as arbitrarily large powers of $\zeta$, we see that $\{0,1\}_N[\zeta]$ is dense in $\RR$ for any $N\in\NN$. Then obviously $\{0,1\}_N[\zeta]$ is dense also in $\zeta\RR$, hence in $\CC$, too.
\end{proof}

\begin{lemma}\label{lemma:circle_group_Z_approx_density_comp}
Let $\zeta\in S$ be different from $\pm 1$. Then $\{0,1\}_N[\zeta]$ is a $1$-net of $\CC$ for any $N\in\NN$.
\end{lemma}
\begin{proof}
If $\zeta\neq \pm i,\pm\omega,\pm\omega^2$, then the statement is obvious from Lemma~\ref{lemma:circle_group_ae_Z_density}. Otherwise, we may assume $N=0$, and if $\zeta=\pm i$ (resp. $\zeta=\pm \omega$ or $\zeta=\pm \omega^2$), then $\ZZ[\zeta]\subset\CC$ is the lattice of Gaussian (resp. Eulerian) integers, which are known from elementary geometry to satisfy that for any $w\in\CC$, there exist $a_0,a_1\in\ZZ$ such that
\[
\left|w-(a_0+a_1\zeta)\right|<1.
\]
Again, the potientially negative coefficients can be switched to a sum of positive ones by recording
\[
-1=\sum_{j=1}^{11} \zeta^j,\qquad -\zeta=\sum_{j=2}^{12} \zeta^j,
\]
and repetitions can be treated via $\zeta^k=\zeta^{k+12}$.
\end{proof}

The following three lemmata serve as a preparation to the proof of Theorem \ref{thm:all_arguments}.

\begin{lemma} \label{lemma:get_around}
If $\Lambda \subseteq \CC$ is not contained by a closed real half-plane of the form $\{z: \text{ } \alpha \leq \arg z \leq \alpha + \pi\}$, then we can find $\Delta_0=\Delta_4, \Delta_1, \Delta_2, \Delta_{3}$ elements of $\Lambda$ such that
\begin{equation} \label{eq:consecutive_arguments}
0\leq \arg{\frac{\Delta_{j+1}}{\Delta_{j}}}<\pi
\end{equation}
for any $0\leq j \leq 3$.
\end{lemma}

\begin{proof}
Multiplying $\Lambda$ by a nonzero scalar does not change the assumption, nor the implication. Consequently, we can assume $1\in \Lambda$. Let $\Delta_0 = 1$. By the condition on $\Lambda$, $\{\lambda: \Im\lambda>0\}$ is nonempty. Consequently, we can define
$$\beta:=\sup_{\lambda\in \Lambda, \text{ } \Im\lambda>0} \arg{\lambda}.$$
Now by the same argument, $\{\lambda\in \Lambda: \text{ } \beta < \arg \lambda < \beta + \pi\}$ is nonempty. However, due to the definition of $\beta$, we can deduce that $\{\lambda\in \Lambda: \text{ } \pi \leq \arg \lambda < \beta + \pi\}$ is nonempty. Let $\Delta_2$ be an element of it. Due to the definition of $\beta$, we can find $\Delta_1$ such that \eqref{eq:consecutive_arguments} is satisfied for $j=0, 1$. Now if $\arg{\Delta_2} > \pi$, the choice $\Delta_3 = \Delta_2$ guarantees that it is also satisfied for $j=2, 3$ and we are done. Otherwise we can choose $\Delta_3$ to be any element of the necessarily nonempty $\{\lambda\in \Lambda, \text{ } \Im\lambda<0\}$, which yields \eqref{eq:consecutive_arguments} for $j=2, 3$.
\end{proof}

\begin{lemma} \label{lemma:reduce_absolute_value}
Assume that $\Lambda \subseteq \CC$ is not contained by a closed real half-plane of the form $\{z: \text{ } \alpha \leq \arg z \leq \alpha + \pi\}$. Then there exists an appropriate $R>0$ with the property that for any $z\in\CC$ satisfying $|z|>R$, we can find $\lambda\in \Lambda$ such that $|z+\lambda|<|z|$.
\end{lemma}

\begin{proof}
Fix $\Delta_0, \Delta_1, \Delta_2, \Delta_{3}$ as guaranteed by Lemma \ref{lemma:get_around}. Let 
$$\alpha_0:=\max_{0\leq j \leq 3}\arg{\frac{\Delta_{j+1}}{\Delta_{j}}}<\pi.$$
Now, if $z \neq 0$ is arbitrary, we can find $0\leq j \leq 3$ such that for $\Delta_j$ we have $\frac{\pi+\alpha_0}{2} \leq \arg \frac{\Delta_j}{z} < \frac{3\pi-\alpha_0}{2}$. Consequently, if we consider the triangle determined by $0, z, z+\Delta_j$, we have that the angle at $z$ is smaller than the right angle, and its size is bounded away from $\frac{\pi}{2}$ by some positive quantity. As the set of all $\Delta_j$'s is bounded, this implies that if $|z|\geq R$ for large enough $R$, then the side $[0,z]$ of the triangle is larger than the side $[0,z+\Delta_j]$. Defining $R$ accordingly proves the lemma.
\end{proof}

Note that if $R$ is sufficient for some $\Lambda$ in the setup of Lemma \ref{lemma:reduce_absolute_value}, then $cR$ is sufficient for $c\Lambda$. In particular, if $c<1$, the same $R$ can be used.

\begin{lemma} \label{lemma:small_absolute_value}
Assume that $\Lambda \subseteq \CC$ is not contained by a closed real half-plane of the form $\{z: \text{ } \alpha \leq \arg z \leq \alpha + \pi\}$. Then there exists $R^*>0$ such that for any $z$, $|z|<1$, and $0\leq n_0<n_1< \ldots$, there exists $(\lambda_{n_j})_{j=0}^{\infty}$, $\lambda_{n_j}\in \Lambda$ such that $\left|\sum_{j=0}^{\infty} \lambda_{n_j}z^{n_j}\right| \leq R^*$.
\end{lemma}

\begin{proof}
We prove that $R^*=R+ \sup_{\lambda\in\Lambda} |\lambda|$ is sufficient, where $R$ is the one guaranteed by Lemma \ref{lemma:reduce_absolute_value}. Due to the note following its proof, the same $R$ can be used for any coefficient set of the form $z^n \Lambda$.

For $z=0$, the claim is trivial, regardless of the choice of $(\lambda_{n_j})_{j=0}^{\infty}$. Hence fix $z\neq 0$, $|z|<1$. The proof depends on a recursive construction of the sequence $(\lambda_{n_j})$. Notably, let $\lambda_{n_0}\in\Lambda$ be arbitrary, and assume $\lambda_{n_0}, \ldots, \lambda_{n_k}$ are already defined. If
$$\left|\sum_{j=0}^{k}\lambda_{n_j}z^{n_j}\right|\leq R,$$
then $\lambda_{n_{k+1}}\in\Lambda$ can be chosen arbitrarily as well. Otherwise, we apply Lemma~\ref{lemma:reduce_absolute_value} to $z^{n_{k+1}}\Lambda$ to define $\lambda_{n_{k+1}}$ such that
$$\left|\sum_{j=0}^{k+1}\lambda_{n_j}z^{n_j}\right|<\left|\sum_{j=0}^{k}\lambda_{n_j}z^{n_j}\right|.$$
These choices obviously guarantee that 
$$\left|\sum_{j=0}^{k}\lambda_{n_j}z^{n_j}\right| \leq R+ \max_{\lambda\in\Lambda} |\lambda| = R^*,$$
regardless of the value of $k$. Consequently, the same bound holds for the sum of the series as well.
\end{proof}

\section{Proof of Theorem~\ref{thm:not_-1_1_arguments}}

Assume first that $0,1\in\Lambda$. For any fixed $w\in\CC$ and $\eps>0$, we introduce
\[
A_{w,\eps}:=\{\bm{\lambda}=(\lambda_n)_{n\in\NN}:\text{there exists some $\tau\in U$ such that $|f_{\bm{\lambda}}(\tau)-w|<\eps$}\}.
\]
Fixing $w\in\CC$ and $\eps>0$, we introduce the abbreviation $A:=A_{w,\eps}$, and prove below that it is open and dense in $\Omega$.

To see that $A$ is open in $\Omega$, let $\bm{\lambda}=(\lambda_n)_{n\in\NN}\in A$, i.e. for some $\tau\in U$, $\eps_0:=|f_{\bm{\lambda}}(\tau)-w|<\eps$. Let $N$ be large enough to satisfy that
\[
\sum_{n>N} \sup\{|\lambda|:\lambda\in\Lambda\} \tau^n < \frac{\eps-\eps_0}{2}.
\]
Also, choose $\delta>0$ in such a way that if $|\lambda_n'-\lambda_n|<\delta$ for all $n\leq N$, then
\[
\left|\sum_{n\leq N} \lambda_n \tau^n - \sum_{n\leq N} \lambda_n' \tau^n\right|< \frac{\eps-\eps_0}{2}.
\]
Clearly, if
\[
\bm{\lambda}'=(\lambda_n')_{n\in\NN} \in \bigtimes_{n\leq N}\{\lambda_n':|\lambda_n'-\lambda_n|<\delta\} \times \bigtimes_{n>N} \Lambda\subseteq \Omega,
\]
then
\[
|f_{\bm{\lambda}'}(\tau)-w|<\eps,
\]
hence $\bm{\lambda}'\in A$, which shows that $A$ is open.

Now we prove that $A$ is dense in $\Omega$. It suffices to show that $A$ intersects any set of the form \[G:=\{\lambda_0\}\times\ldots\times\{\lambda_N\}\times \bigtimes_{n>N} \Lambda\subseteq \Omega.\] Let us fix some $\pm 1\neq \zeta\in \overline{U} \cap S$ throughout the proof. Our goal is to find an element $\bm{\lambda}=(\lambda_n)_{n\in\NN}\in G$ and some $\tau\in U$ such that $|f_{\bm{\lambda}}(\tau)-w|<\eps$. We immediately prescribe $|\tau-\zeta|<\delta$ with $0<\delta<2/5$ chosen in such a way that
\[
\left|\sum_{n\leq N} \lambda_n \tau^n - \sum_{n\leq N} \lambda_n \zeta^n \right|<\frac{\eps}{2}
\]
is guaranteed by $|\tau-\zeta|<\delta$.

Relabeling our original $w$ (shifting it by $- \sum_{n\leq N} \lambda_n \zeta^n$), and rescaling $\eps$, we have to find a sequence $(\lambda_n)_{n>N}\in \{0,1\}^{n>N}$ and some $\tau\in U$ in such a way that
\begin{equation}\label{eq:pf_thm_1_density_goal}
\left|\sum_{n>N} \lambda_n \tau^n - w\right|<\eps,\qquad |\tau-\zeta|<\delta.
\end{equation}
Now we go by cases according to Lemmata~\ref{lemma:circle_group_ae_Z_density}--\ref{lemma:circle_group_Z_approx_density_comp} about the nature of $\zeta$. To avoid notational difficulties, assume that $w\neq 0$ (which is not a real restriction, since if anything but zero can be arbitrarily approximated, then so can zero).

If $\zeta\neq \pm i,\pm \omega, \pm \omega^2$, then by Lemma~\ref{lemma:circle_group_ae_Z_density}, we may find and fix a \emph{finite} sequence $(\lambda_n)_{N<n<K}$ satisfying
\[
\left|\sum_{N<n<K} \lambda_n \zeta^n - w\right|<\frac{\eps}{2}, \qquad \lambda_n\in\{0,1\},
\]
and then if $|\tau-\zeta|$ is small enough (consistently with the earlier prescribed $|\tau-\zeta|<\delta$),
\[
\left|\sum_{N<n<K} \lambda_n \tau^n - \sum_{N<n<K} \lambda_n \zeta^n\right|<\frac{\eps}{2}, \qquad |\tau-\zeta|<\delta,
\]
which, setting $\lambda_K:=\lambda_{K+1}:=\ldots:=0$, together clearly imply \eqref{eq:pf_thm_1_density_goal}.

Now assume that $\zeta\in\{\pm i,\pm \omega, \pm\omega^2\}$. Fix a finite set $W$ which is a $1$-net of $D_{10|w|/\eps}$. By Lemma~\ref{lemma:circle_group_Z_approx_density_comp}, for any $w'\in W$, we may find and fix a \emph{finite} sequence $(\lambda_n)_{N<n<K}(w')$ satisfying
\[
\left|\sum_{N<n<K} \lambda_n(w') \zeta^n - w'\right|<1, \qquad \lambda_n(w')\in\{0,1\},
\]
where the upper bound $K$ on the coefficient indices is uniform over $w'\in W$ (this can be achieved, since $W$ is finite). Now if $|\tau-\zeta|$ is small enough (consistently with the earlier prescribed $|\tau-\zeta|<\delta$), then
\[
\left|\sum_{N<n<K} \lambda_n(w') \tau^n - \sum_{N<n<K} \lambda_n(w') \zeta^n\right|<1\text{ for all $w'\in W$}, \qquad |\tau-\zeta|<\delta.
\]
Fix now $M$ in such a way that $\eps/5<|\tau^M|<\eps/3$ (this is possible, since $\delta<2/5$ implies $|\tau|>3/5$, and that $\tau$ has a power in the indicated annulus). Then
\[
\left\{\sum_{N<n<K} \lambda_n(w') \tau^{n+M}:w'\in W\right\}
\]
gives rise to an $\eps$-net of $D_{2|w|}$. In particular, choosing the appropriate $w'$ (for which the sum in the last display is closest to $w$), and setting
\[
\lambda_n:=
\begin{cases}
\lambda_n,\qquad &\text{if $n\leq N$,}\\
\lambda_{n-M}(w'),& \text{if $N+M<n<K+M$,}\\
0& \text{otherwise,} 
\end{cases}
\]
\eqref{eq:pf_thm_1_density_goal} is achieved.

To sum up, any $A_{w,\eps}$ is open and dense, which in turn implies that the set
\[
\bigcap_{w\in \QQ+\QQ i} \bigcap_{k=1}^{\infty} A_{w,1/k}
\]
is residual, hence the proof of Theorem~\ref{thm:not_-1_1_arguments} is complete, at least, when $0,1\in\Lambda$. This, however, immediately gives rise to the general case by applying Lemma~\ref{lemma:translate_Lambda}~\ref{lemma:translate_mult}--\ref{lemma:translate_add_shrink}. Indeed, the proof of the density of $A$ presented above guarantees $\tau$'s arbitrarily close to $\zeta$, which means that the condition of Lemma~\ref{lemma:translate_Lambda}~\ref{lemma:translate_add_shrink} is satisfied.

A fairly straightforward consequence of Theorem \ref{thm:not_-1_1_arguments} is the following:

\begin{corollary}
For the generic $(\lambda_n)_{n=0}^{\infty}\in\Omega$, for any $\zeta \in S$ and $w\in\CC$ there exists $(\zeta_k)_{k=1}^{\infty}\subseteq D$ with $\zeta_k\to \zeta$ such that $f(\zeta_k)\to w$.
\end{corollary}

The proof is left as a simple exercise to the reader, it suffices to rely on the case of irrational arguments. With a slightly different formulation, proving this statement was Problem 9 at the prestigious Miklós Schweitzer Memorial Competition for Hungarian university students in 2020, proposed by the authors. Complete solutions were given by Márton Borbényi and Attila Gáspár, who were awarded the two first prizes of the contest, and to whom we congratulate hereby. A direct solution is available at \cite{Schw} in Hungarian.

\section{Proof of Theorem \ref{thm:not_1_arguments}}

Due to the open mapping theorem, it suffices to prove that the density of $f(U)$ holds generically.

Due to Lemma~\ref{lemma:translate_Lambda}~\ref{lemma:translate_mult}--\ref{lemma:translate_add_shrink}, we can assume $0, 1 \in \Lambda$. First we will consider the case when $\Lambda$ is contained by a real line. Following our assumption, this means $\Lambda\subseteq \mathbb{R}$.

We will define $U=\bigcup_{k=1}^{\infty}U_k$ where 
$$U_k=\left\{z: -\frac{k-1}{k} < \Re z <0, \text{ } \pi-\alpha_k < \arg{z} < \pi+\alpha_k \right\}$$
for $\alpha_k>0$ to be fixed later. (The lower bound on $\Re z$ is somewhat arbitrary, separation from $-1$ is relevant only.) Right now we specify only that $\alpha_k$ is small enough to guarantee that $\overline{U_k}\subseteq D$. 

As each $U_k$ is open, the same holds for $U$. Now consider any $z\in U_k$. As $\overline{U_k}\subseteq S$, we can choose $N$ large enough to have 
\begin{equation} \label{eq:N_choice}
\left|\sum_{n=2N}^{\infty}z^n\right|<1.
\end{equation}
We choose $\alpha_k$ based on the choice of $N$ such that
$$2N\alpha_k<\arcsin\left(\frac{1}{N}\right).$$
This clearly implies that 
$$-\arcsin\left(\frac{1}{N}\right)<\arg{\left(\sum_{n=0}^{N-1}z^{2n}\right)}<\arcsin\left(\frac{1}{N}\right)$$
and
$$\pi-\arcsin\left(\frac{1}{N}\right)<\arg{\left(\sum_{n=0}^{N-1}z^{2n+1}\right)}<\pi+\arcsin\left(\frac{1}{N}\right).$$
However, the absolute value of each of these partial sums is at most $N$, which yields that their imaginary part is at most $1$. Consequently, 
$$\Im\left(\sum_{n=0}^{2N-1}z^n\right)\leq 2.$$
By the same argument for any $(\lambda_n)_{n=0}^{\infty}$ we have
$$\Im\left(\sum_{n=0}^{2N-1}\lambda_nz^n\right)\leq 2\sup_{\lambda\in \Lambda}|\lambda|.$$
Taking \eqref{eq:N_choice} into consideration implies
$$\Im\left(\sum_{n=0}^{\infty}\lambda_nz^n\right)\leq 3\sup_{\lambda\in \Lambda}|\lambda|.$$
As it holds for any $k$ and $z\in U_k$, we have it for any $z\in U$, which concludes the proof of the first part.

In the other direction, our argument will be similar to the one we have given in the proof of Theorem \ref{thm:not_-1_1_arguments}. Indeed, defining the set $A$ as above, and following the argument verbatim, it suffices to find a sequence $(\lambda_n)_{n>N}\in \{0,1\}^{n>N}$ and some $\tau\in U$ in such a way that



\[
\left|\sum_{n>N} \lambda_n \tau^n - w\right|<\eps,\qquad |\tau-(-1)|<\delta.
\]

Fix an element $\lambda\in \Lambda$ with non-vanishing imaginary part; its existence is guaranteed by the assumption that $0, 1 \in \Lambda$ and $\Lambda\nsubseteq \mathbb{R}$. Consider the lattice
\[\{a+b\lambda: \text{ } a,b\in\mathbb{Z}\}.\]
It obviously gives a $\delta$-net of $\CC$ for some $\delta>0$.  Consequently,
\[\{\xi (a+b\lambda): \text{ } a,b\in\mathbb{Z}\}\]
gives an $\eps/2$-net of $\CC$ for $|\xi|=\varepsilon_0$ if $\varepsilon_0>0$ is small enough. Fix $\varepsilon_0$ accordingly, and fix $R$ such that $|w|<\varepsilon_0 R$. Now it is clear that one can find $m_R$ such that for
\[C^{m_R}=\{a+b\lambda: \text{ } a,b\in\mathbb{Z}, \text{ } |a|, |b|<m_R\},\]
$\xi C^{m_R}$ gives an $\eps/2$-net of $D_R$ for $|\xi|=\varepsilon_0$.

Now let us notice that for large enough $M_R$, any element of $C^{m_R}$ can be written in the form \[\lambda\sum_{j=1}^{k}(-1)^{n_j}+\sum_{j'=1}^{l}(-1)^{n'_{j'}},\] where the exponents used are pairwise distinct, and $N<n_j, n'_{j'} <M_R$ for $j=1, \ldots, k$ and $j'=1, \ldots, l$. Denote the set of such sums by $C(-1)$, and motivated by this, let
\[C(z)=\left\{\lambda\sum_{j=1}^{k}z^{n_j}+\sum_{j'=1}^{l}z^{n'_{j'}}: \text{ } N<n_j, n'_{j'} <M_R \text{ are distinct}\right\}.\]
As $C(z)$ is determined by finitely many continuous functions of $z$, if $|(-1) - \tau|$ is small enough, $\xi C(\tau)$ gives a $\varepsilon$-net of $D_{\varepsilon_0 R}$ for any $\xi, \text{ } |\xi|=\varepsilon_0$. On the other hand, we can choose $\tau$ in any neighborhood of $-1$ so that $|\tau|^M= \varepsilon_0$ for some $M>0$. Consequently, we have that $\tau^M C(\tau)$ forms a $\varepsilon$-net of $D_{\varepsilon_0 R}$.

By definition, this implies that there exist pairwise distinct numbers $N<n_j, n'_{j'} <M_R$ for $j=1, ..., k$, $j'=1, ..., l$ such that 
$$\left|\left(\lambda\sum_{j=1}^{k}\tau^{n_j+M}+\sum_{j'=1}^{l}\tau^{n'_{j'}+M}\right) - w\right| < \varepsilon.$$
Now if we define $(\lambda_n)_{n=N+1}^{\infty}$ such that $\lambda_n=\lambda$ if and only if $n=n_j+M$ for some $1\leq j \leq k$, moreover, $\lambda_n=1$ if and only if $n=n'_{j'}+M$ for some $1\leq j' \leq l$, and otherwise $\lambda_n=0$, then we immediately obtain $|f_{\bm{\lambda}}(\tau) - w| < \varepsilon$, which concludes the proof.

\section{Proof of Theorem \ref{thm:all_arguments}}

Due to the open mapping theorem, it suffices to prove that the density of $f(U)$ holds generically.

First we will consider the case when $\Lambda$ is contained by a closed real half-plane of the form $\{z: \text{ } \alpha \leq \arg z \leq \alpha + \pi\}$. Due to Lemma \ref{lemma:translate_Lambda}~\ref{lemma:translate_mult}, we can assume that this half-plane is $\{ \Re z \geq 0\}$.

We will define $U=\bigcup_{k=1}^{\infty}U_k$ where 
$$U_k=\left\{z: 0 < \Re z <  \frac{k-1}{k}, \text{ } -\alpha_k < \arg{z} < \alpha_k \right\}$$
for $\alpha_k>0$ to be fixed later. (The upper bound on $\Re z$ is somewhat arbitrary, separation from $1$ is relevant only.) Right now we specify only that $\alpha_k$ is small enough to guarantee that $\overline{U_k}\subseteq S$.

From this point, the proof of this part is basically a simplified version of the proof of the same part of the proof of Theorem \ref{thm:not_1_arguments}. Notably, for $z\in U_k$ we can find a threshold index such that the tail sum is very small due to $|z|$ being bounded away from 1, and then by choosing $\alpha_k$ to be small enough, we can control the argument of the preceding terms. (The relative simplicity in this case is due to the fact these arguments are all near 0, instead of being near to 0 and $\pi$ alternatingly.) Based on these estimates, the real part of $f_{\bm{\lambda}}(z)$ can be bounded from below, regardless of $\bm{\lambda}=(\lambda_n)_{n=0}^{\infty}$ and $z\in  U_k$, which concludes the proof of the first part.

We now prove the second part of the statement of the theorem. Defining the set $A$ as in the proof of Theorem \ref{thm:not_-1_1_arguments} and following the argument verbatim, it suffices to find a sequence $(\lambda_n)_{n>N}\in \Lambda^{n>N}$ and some $\tau\in U$ in such a way that
\[
\left|\sum_{n>N} \lambda_n \tau^n - w\right|<\eps,\qquad |\tau-1|<\delta.
\]

Define $\Delta_0, \Delta_1, \Delta_2, \Delta_{3}\in\Lambda$ as guaranteed by Lemma \ref{lemma:get_around}. Denote their set by $V$, and define the convex polygon $P = \text{conv}(V)$. Due to the choice of $V$, $0\in \operatorname{int}(P)$, that is $D_r\subseteq V$ for small enough $r$.

Notice that if $P$ has diameter $\delta$, then $V$ is clearly a $\delta$-net of $P$. Moreover, for any $m$ we have that the Minkowski sum $\sum_{i=1}^{m} V$ is a $\delta$-net of $\sum_{i=1}^{m} P$: the proof proceeds by induction, capitalizing on the simple observation that $\sum_{i=1}^{m} P = V + \sum_{i=1}^{m-1} P$. It clearly yields that $\sum_{i=1}^{m} V$ is a $\delta$-net of $D_{mr}$ as well for any $m$. Consequently, $\xi\cdot \sum_{i=1}^{m} V$ gives an $\eps/2$-net of $D_{\eps_0 m r}$ for any $m$ and for $|\xi|=\varepsilon_0$, if $\varepsilon_0>0$ is small enough. Fix $\varepsilon_0$ accordingly, noting that it does not depend on $m$. Now fix $R^*$ as guaranteed by Lemma \ref{lemma:small_absolute_value}, and based on the choice of $R^*$ and $\varepsilon_0$, fix $m$ such that $|w| + R^*<\varepsilon_0 m r$.
Let us remark that any element of $\sum_{i=1}^{m} V$ is expressible in the seemingly complicated form
$$\sum_{j=0}^{3}\Delta_j k_j = \sum_{j=0}^{3}\Delta_j \sum_{l=1}^{k_j}1^{n^{(j)}_l},$$
where each $0\leq k_j\leq m$, and the exponents $n^{(j)}_l$ are pairwise distinct and their union equals $\{N+1, N+2, ..., N+m\}$. Let us denote the set of such combinations by $C(1)$, and motivated by this, let
\[C(z)=\left\{\sum_{j=0}^{3}\Delta_j \sum_{l=1}^{k_j}z^{n^{(j)}_l}: \text{ } 0\leq k_j \leq m, \text{ } N< n^{(j)}_l \leq N+m \text{ are all distinct and their union is }\{N+1, N+2, ..., N+m\}\right\}.\]
As $C(z)$ is determined by finitely many continuous functions of $z$, if $|1 - \tau|$ is small enough, $\xi C(\tau)$ gives a $\varepsilon$-net of $D_{\varepsilon_0 m r}$ for any $\xi, \text{ } |\xi|=\varepsilon_0$. On the other hand, we can choose $\tau$ in any neighborhood of $1$ so that $|\tau|^M= \varepsilon_0$ for some $M>0$. Consequently, we have that $\tau^M C(\tau)$ forms an $\varepsilon$-net of $D_{\varepsilon_0 m r}$.

So far the coefficients of the power series we would like to define are fixed for the indices $(i)_{i=0}^{N}$. Motivated by the previous paragraph, we would like to set aside the indices $(N+i+M)_{i=1}^{m}$. Notably, these are the indices which are intimately connected to the lastly defined $\tau^M C(\tau)$. Consequently, we apply Lemma \ref{lemma:small_absolute_value} at this point for $\tau$ and the complementary sequence $(N+1, N+2, ..., N+M, N+m+M+1, N+m+M+2, ...)$: we can find elements of $\Lambda$ corresponding to these indices, $(\lambda_n)_{n=N+1}^{N+M}$, $(\lambda_n)_{n=N+m+M+1}^{\infty}$
such that 
$$|w_1|\leq R^*,\qquad \text{where}\qquad w_1:=\sum_{n=N+1}^{N+M} \lambda_{n}z^{n} + \sum_{n=N+m+M+1}^{\infty} \lambda_{n}z^{n}.$$
Consequently, $|w-w_1|\leq |w|+R^*<\varepsilon_0 R$. This implies that there exist numbers 
\[N< n^{(j)}_l\leq N+m\] 
for $j=0, 1, 2, 3$ and $l=1, ..., k_j$, such that their union fills $\{N+1, N+2, ..., N+m\}$ without any repetitions (that is the numbers $n^{(j)}_l$ are pairwise distinct for all the possible choices of $j, l$), and
\begin{equation} \label{eq:approx+w1}
\left|\sum_{j=0}^{3}\Delta_j \sum_{l=1}^{k_j}\tau^{n^{(j)}_l + M} - (w-w_1)\right| < \frac{\varepsilon}{2}.
\end{equation}
What remains from the definition of $(\lambda_n)$ is fixing $(\lambda_n)_{n=N+1+M}^{N+m+M}$, which we carry out now based on \eqref{eq:approx+w1}. Notably let $\lambda_n=\Delta_j$ if and only if $n=n^{(j)}_l+M$ for some $1\leq l \leq k_j$. Then by the definition of $w_1$ we obtain $|\sum_{n=N+1}^{\infty}\lambda_n\tau'^n - w| < \varepsilon$. This concludes the proof.

\section{Concluding remarks}

Even though with a careful separation of cases we managed to generalize the most natural result given by Theorem \ref{thm:not_-1_1_arguments}, our results are far from being complete. In our view, the most interesting open problem related to them is whether $f(U)=\CC$ holds generically in the setup of our theorems, similarly to what is proved in \cite{Kah}. As $f$ is uniformly locally bounded in $D$, we clearly cannot rely on techniques similar to the ones seen there, hence answering this question requires additional ideas.

Another interesting aspect partially inspired by this paper is whether we can rearrange the quantifiers in our statements to some extent. More explicitly, each of our theorems addresses the question of what the generic image is of a {\it fixed} open set. It would be desirable to find extensions of this result, for example a nontrivial family of open sets such that generically, the image of each of them is dense.

\end{document}